\documentclass[11pt,letterpaper,reqno]{amsart}
\usepackage{amsmath,amssymb}
\usepackage{mathrsfs}
\usepackage{amsfonts}
\usepackage[dvipsnames]{xcolor}
\usepackage{graphicx}
\usepackage{float}  
\usepackage{cite}
\usepackage{cases}
\usepackage{comment}
\usepackage{esint}
\usepackage{indentfirst}
\allowdisplaybreaks

\theoremstyle{plain}
\newtheorem{theorem}{Theorem}[section]
\newtheorem{lemma}[theorem]{Lemma}
\newtheorem{proposition}[theorem]{Proposition}
\newtheorem{corollary}[theorem]{Corollary}

\theoremstyle{definition}

\theoremstyle{remark}

\numberwithin{equation}{section}

\DeclareMathOperator{\tr}{tr}
\DeclareMathOperator{\osc}{osc}

\DeclareMathOperator{\dist}{dist}
\newcommand{\ac}{\mathrm{ac}}
\newcommand{\sing}{\mathrm{s}}
\newcommand{\Sym}{\operatorname{Sym}(n)}

\usepackage[colorlinks=true,citecolor=blue,linkcolor=blue]{hyperref}

\begin{document}
\title[Interior $W^{2,p}$ regularity]
{Interior $W^{2,p}$ regularity for the sigma-$2$ equation}

\author[Y. Lyu]{Yasheng Lyu}
\address{School of Mathematics and Statistics, Xi'an Jiaotong University, Xi'an, Shaanxi 710049, People's Republic of China}
\email{lvysh21@stu.xjtu.edu.cn}

%**********************************************************************************
\date{}
%\subjclass[2020]{Primary 35J96; Secondary 35B45, 35B65}

\begin{abstract}
We prove that continuous $2$-convex solutions of the $\sigma_{2}$ equation with a density bounded above and below by positive constants belong to $W^{2,1}_{\mathrm{loc}}$ in every dimension.
If, in addition, the density is continuous, the solutions belong to $W^{2,p}_{\mathrm{loc}}$ for every $1<p<\infty$.
\end{abstract}

\maketitle

\section{Introduction}

In this paper, we study interior $W^{2,p}$ regularity for the $2$-Hessian equation
\begin{equation}\label{eqn1.1}
\sigma_2(D^2u)=f\quad\text{in }\Omega,
\end{equation}
where $\Omega\subset\mathbb R^n$ is a domain and $n\geq2$.
Let $\Sym$ denote the space of real symmetric $n\times n$ matrices.
We use the Frobenius inner product $A:B=\tr(AB)$ and its norm $|A|$.
For $H\in\Sym$, set
\[
\sigma_2(H):=\frac12\bigl((\tr H)^2-|H|^2\bigr),\qquad
\Gamma_2:=\{H:\tr H>0,\ \sigma_2(H)>0\}.
\]
Let $\Phi_2(\Omega)$ denote the class of $2$-convex functions, recalled in
Section~\ref{sec-measures}.

We formulate~\eqref{eqn1.1} using the Hessian measures of Trudinger and
Wang~\cite{TW1,TW2}: a continuous $2$-convex function is a solution if
\[
\mu_2[u]=f\,dx,
\]
where $\mu_2[u]$ is its $2$-Hessian measure, extending
$\sigma_2(D^2u)\,dx$ from smooth functions. For continuous positive $f$,
this agrees with the admissible viscosity formulation.

\begin{samepage}
\begin{theorem}\label{thm1.1}
Let $\Omega\subset\mathbb R^n$ be a domain, $n\geq2$, and let
$u\in C^0(\Omega)\cap\Phi_2(\Omega)$ satisfy
\begin{equation}\label{eqn1.3}
\begin{gathered}
\mu_2[u]=f\,dx,\qquad f\in L^\infty(\Omega),\\
0<\lambda\leq f\leq\Lambda<\infty\quad\text{almost everywhere in }\Omega.
\end{gathered}
\end{equation}
Then $u\in W^{2,1}_{\mathrm{loc}}(\Omega)$.
For every $B_{2r}(x_0)\Subset\Omega$,
\begin{equation}\label{w21-interior-estimate}
\|D^2u\|_{L^1(B_r(x_0))}
\leq C_n r^{n-2}\osc_{B_{2r}(x_0)}u,
\end{equation}
where $C_n$ depends only on $n$.
\end{theorem}
\end{samepage}

\begin{theorem}\label{thm-w2p}
Let $\Omega\subset\mathbb R^n$ be a domain, $n\geq2$, and let
$u\in C^0(\Omega)\cap\Phi_2(\Omega)$ satisfy
\[
\mu_2[u]=f\,dx,\qquad f\in C^0(\Omega),\qquad
0<\lambda\leq f\leq\Lambda<\infty.
\]
Then $u\in W^{2,p}_{\mathrm{loc}}(\Omega)$ for every $1<p<\infty$.
For every $B_{2r}(x_0)\Subset\Omega$,
\begin{equation}\label{wp-interior-estimate}
r^{-n/p}\|D^2u\|_{L^p(B_r(x_0))}\leq C,
\end{equation}
where $C$ depends only on $n$, $p$, $\lambda$, $\Lambda$,
$r^{-2}\osc_{B_{2r}(x_0)}u$, and the modulus of continuity of the
rescaled density $\widetilde f(y)=f(x_0+ry)$ on $B_2$.
\end{theorem}

The conclusion of Theorem~\ref{thm1.1} also holds under the admissible
viscosity inequalities
\begin{equation}\label{eqn1.4}
\lambda\leq\sigma_2(D^2u)\leq\Lambda\quad\text{in }\Omega,
\end{equation}
which imply the measure hypothesis~\eqref{eqn1.3}; see
Lemma~\ref{thm2.4}.

For the Monge--Amp\`ere equation, Caffarelli~\cite{Caffarelli} proved
interior $W^{2,p}$ estimates for strictly convex solutions with continuous
positive density. For densities bounded above and below by positive
constants, De Philippis and Figalli~\cite{DF} established
$W^{2,1}_{\mathrm{loc}}$ regularity for strictly convex solutions.
De Philippis, Figalli, and Savin~\cite{DFS}, and independently
Schmidt~\cite{Schmidt}, improved this to $W^{2,1+\varepsilon}_{\mathrm{loc}}$.
Mooney~\cite{Mooney} removed strict convexity at the $W^{2,1}$ endpoint
and subsequently obtained a logarithmic improvement~\cite{MooneyEstimate}.
For $n\geq3$, his constant-density examples show that
$W^{2,1+\varepsilon}_{\mathrm{loc}}$ regularity cannot hold in general for
any $\varepsilon>0$.
In dimension two, $\sigma_2=\det$ and $2$-convexity is convexity;
positive bounded densities imply strict convexity, so both theorems
follow from this theory.

For the quadratic Hessian equation, the constant-density interior
Hessian estimate was proved by Warren and Yuan~\cite{WY} in dimension
three and Shankar and Yuan~\cite{SY25} in dimension four.
In higher dimensions, Guan and Qiu~\cite{GQ}, McGonagle, Song, and
Yuan~\cite{MSY}, and Shankar and Yuan~\cite{SY20} obtained estimates for
smooth convex, almost convex, and semiconvex solutions, respectively.
Mooney~\cite{MooneyStrict} and Shankar and Yuan~\cite{SY21} established
regularity in the convex and almost convex viscosity classes.
Li and Wu~\cite{LW} proved smoothness and interior Hessian bounds for
continuous admissible viscosity solutions of $\sigma_2=1$ in every
dimension, with constants depending only on the dimension and the
$L^\infty$ norm of the solution.

For variable densities in dimension three, Qiu~\cite{Qiu} proved Hessian
estimates for positive $C^2$ densities, Xu~\cite{Xu} obtained perturbative
$C^{2,\alpha}$ estimates under a small H\"older seminorm condition,
and Zhou~\cite{Zhou24} established $C^{2,\alpha}$ regularity for positive
Lipschitz densities. Fan~\cite{Fan} treated positive $C^{1,1}$ densities
in dimension four.
In arbitrary dimensions, Chen, Jian, and Zhou~\cite{CJZ} obtained
Hessian estimates for smooth convex solutions with constants depending
on the density through its positive lower bound and Lipschitz norm.
Chen, Jian, Tu, and Zhou~\cite{CJTZ} proved $C^2$ regularity for convex
viscosity solutions with positive Lipschitz density; their argument uses
weighted $W^{2,p}$ estimates for smooth admissible solutions.
For positive $C^\alpha$ densities, Zhou and Zhu~\cite{ZZ} proved
$C^{2,\alpha}$ regularity for convex viscosity solutions, and
Chen, Zhou, and Zhu~\cite{CZZ} treated the full admissible viscosity class.

Related integrability results include almost everywhere second-order
differentiability for Lipschitz admissible solutions with bounded density
in dimensions $n\geq4$~\cite{Fan} and $W^{2,1}_{\mathrm{loc}}$
regularity for convex solutions of $\sigma_k=1$, with new results
for $3\leq k<n$~\cite{Hu}.
For the $k$-Hessian equation with smooth positive density,
Urbas~\cite{Urbas01} obtained interior Hessian bounds from a priori
$W^{2,p}$ control for $p>k(n-1)/2$.
For smooth solutions of $\sigma_2=1$, Mooney~\cite{MooneyQuadratic}
showed that $p>2$ suffices.

Our proof uses the quadratic comparison identity and trace-dependent
Dirichlet modification of Chen, Zhou, and Zhu~\cite{CZZ}, and the
constant-density estimates of Li and Wu~\cite{LW}.
For measurable densities, we allow signed density errors in the modified
equation and pass the comparison identity to mixed Hessian measures.
A weighted trace estimate then eliminates the singular part of the
distributional Hessian. For continuous densities, a modification depending
on the comparison gap yields a contraction of the weighted trace tails,
which gives every finite $L^p$ exponent.
These arguments occupy Sections~\ref{sec-w21} and~\ref{sec-w2p}, respectively.

\section{Interior \texorpdfstring{$W^{2,1}$}{W2,1} regularity}\label{sec-w21}

\subsection{Hessian measures}\label{sec-measures}

A function $u:\Omega\to[-\infty,\infty)$ is $2$-convex if it is upper
semicontinuous, is not identically $-\infty$ on any connected component
of $\Omega$, and every $C^2$ function touching $u$ from above has its
Hessian in $\overline\Gamma_2$ at the contact point. We denote this class
by $\Phi_2(\Omega)$.
A smooth function is called admissible if its Hessian lies in $\Gamma_2$.
In~\eqref{eqn1.4}, the lower inequality is tested by $C^2$ functions
touching from above; the upper inequality is tested by $C^2$ functions
touching from below whose Hessians belong to $\Gamma_2$ at the contact
point.

The Hessian measure construction~\cite[Theorem~1.1]{TW2} associates to
each $u\in\Phi_2(\Omega)$ a unique nonnegative Radon measure $\mu_2[u]$.
It agrees with $\sigma_2(D^2u)\,dx$ when $u$ is $C^2$ and is weakly
continuous: if $u_j,u\in\Phi_2(\Omega)$ and $u_j\to u$ locally in $L^1$,
then
\[
\int_\Omega\varphi\,d\mu_2[u_j]
\longrightarrow\int_\Omega\varphi\,d\mu_2[u]
\qquad(\varphi\in C_c(\Omega)).
\]

Set
\[
G(H)=\sqrt{\sigma_2(H)}\quad(H\in\Gamma_2),\qquad
N(H)=(\tr H)I-H\quad(H\in\Sym).
\]
The function $G$ is concave and homogeneous of degree one, and $DG(H)=N(H)/(2G(H))$. For $A,B\in\overline\Gamma_2$, define
\begin{equation}\label{eqn2.1}
\mathscr B(A,B)=(\tr A)(\tr B)-A:B,
\qquad \mathscr D(A,B)=\mathscr B(A,B)-2G(A)G(B),
\end{equation}
where $G$ is extended continuously to the closed cone. Concavity and
homogeneity give $\mathscr D(A,B)\geq0$. For $2$-convex functions, the
corresponding mixed measure is
\begin{equation}\label{eqn2.2}
\mathcal B[u,v]=\mu_2[u+v]-\mu_2[u]-\mu_2[v].
\end{equation}
Smooth approximation, polarization, and the preceding weak continuity
show that $\mathcal B$ is nonnegative, symmetric, bilinear, and weakly
continuous under local $L^1$ convergence of both arguments. In particular, for $q(x)=|x|^2/2$,
\begin{equation}\label{eqn2.3}
\mathcal B[u,q]=(n-1)\Delta u.
\end{equation}

For a matrix-valued measure, $|\cdot|$ denotes total variation in the
Frobenius norm; the subscripts $\ac$ and $\sing$ denote the absolutely
continuous and singular parts with respect to Lebesgue measure.

\begin{lemma}\label{thm2.1}
Let $u,v\in\Phi_2(\Omega)$ be locally bounded. Their distributional Hessians are locally finite measures, and
\begin{equation}\label{eqn2.4}
|D^2u|\leq\Delta u,\qquad u\in W^{1,1}_{\mathrm{loc}}(\Omega).
\end{equation}
Write
\[
D^2u=A\,dx+(D^2u)_{\sing},\qquad
D^2v=B\,dx+(D^2v)_{\sing}.
\]
Then $A,B\in\overline\Gamma_2$ almost everywhere, and
\begin{equation}\label{eqn2.5}
\mathcal B[u,v]\geq\mathscr B(A,B)\,dx.
\end{equation}
If $u_j$ are smooth admissible functions converging to $u$ locally in $L^1$, with $G(D^2u_j)\geq m>0$, then
\begin{equation}\label{eqn2.6}
G(A)\geq m\quad\text{almost everywhere}.
\end{equation}
\end{lemma}

\begin{proof}
By \cite[Lemma~2.3]{TW2}, the local convolutions $u_\varepsilon$
are $2$-convex and satisfy $|D^2u_\varepsilon|\leq\Delta u_\varepsilon$.
Since $u$ is subharmonic, $\Delta u$ is a nonnegative Radon measure,
so these Hessians have uniformly bounded variation on compact subsets.
Passing to measures proves $|D^2u|\leq\Delta u$; the Sobolev assertion
in~\eqref{eqn2.4} follows from \cite[Theorem~4.1]{TW2}.
Lebesgue differentiation gives $A,B\in\overline\Gamma_2$ almost everywhere.

Let $u_\varepsilon,v_\varepsilon$ be simultaneous local convolutions. Their Hessians converge almost everywhere to $A,B$, respectively. Since $\mathscr B(D^2u_\varepsilon,D^2v_\varepsilon)\geq0$, Fatou's lemma and weak continuity of \eqref{eqn2.2}, tested against a nonnegative compactly supported function, give \eqref{eqn2.5}.

For each fixed $Q\in\Gamma_2$, concavity gives
\[
DG(Q):D^2u_j\geq G(D^2u_j)\geq m.
\]
Passing to distributions and then taking the absolutely continuous part yields $DG(Q):A\geq m$. These inequalities hold simultaneously for a countable dense subset of $\Gamma_2$. The supporting-plane representation
\[
G(A)=\inf_{Q\in\Gamma_2}DG(Q):A\qquad(A\in\overline\Gamma_2)
\]
proves \eqref{eqn2.6}. This representation follows from concavity; for the reverse inequality take $Q=A+\varepsilon I$ and let $\varepsilon\downarrow0$.
\end{proof}

In particular, $u\in W^{2,1}_{\mathrm{loc}}(\Omega)$ if and only if
$(D^2u)_{\sing}=0$. Thus the absolute continuity of $\mu_2[u]$ in
\eqref{eqn1.3} concerns the scalar Hessian measure, whereas the
conclusion of Theorem~\ref{thm1.1} concerns the distributional Hessian.

We also use the quadratic comparison identity of Chen, Zhou, and
Zhu~\cite[Proposition~3.8]{CZZ}. If $D$ is a ball and $a,v$ are
smooth admissible functions on $\overline D$ with $a=v$ on
$\partial D$, then
\begin{equation}\label{eqn2.8}
\int_D\mathscr D(D^2a,D^2v)
+\frac12\int_{\partial D}H_{\partial D}(a-v)_\nu^2
=\int_D\bigl(G(D^2a)-G(D^2v)\bigr)^2,
\end{equation}
where $\nu$ is the outward unit normal,
$\mathrm{II}_{\partial D}(X,Y)=D_X\nu\cdot Y$ is the second fundamental
form, and $H_{\partial D}=\tr\mathrm{II}_{\partial D}$.
Thus $\mathrm{II}_{\partial D}$ is positive on spheres.
Indeed, for $w=a-v$, polarization gives
\[
\mathscr D(D^2a,D^2v)
=\bigl(G(D^2a)-G(D^2v)\bigr)^2-\sigma_2(D^2w).
\]
The Newton tensor is divergence free, so
\[
2\int_D\sigma_2(D^2w)
=\int_{\partial D}N(D^2w)Dw\cdot\nu
=\int_{\partial D}H_{\partial D}w_\nu^2.
\]
Here $D^2_{\tan}w$ denotes the restriction of the Euclidean Hessian
to the tangent space of $\partial D$. Since $w=0$ on $\partial D$,
we have $Dw=w_\nu\nu$ and $D^2_{\tan}w=w_\nu\mathrm{II}_{\partial D}$
there. In particular, the boundary term in~\eqref{eqn2.8} is nonnegative.

\begin{lemma}\label{thm2.4}
Let $u\in C^0(\Omega)\cap\Phi_2(\Omega)$ satisfy \eqref{eqn1.3}. Fix a ball $D\Subset\Omega$ and write $\mu_2[u]=g^2\,dx$, with $m=\sqrt\lambda\leq g\leq M=\sqrt\Lambda$. There are smooth admissible functions $u_j$ on $\overline D$ such that
\begin{equation}\label{eqn2.9}
\begin{gathered}
G(D^2u_j)=g_j,\qquad m\leq g_j\leq M,\\
g_j\longrightarrow g\text{ in }L^q(D)\quad(1\leq q<\infty),\\
u_j\longrightarrow u\text{ uniformly on }\overline D.
\end{gathered}
\end{equation}
Moreover, the viscosity inequalities \eqref{eqn1.4} imply \eqref{eqn1.3}.
\end{lemma}

\begin{proof}
Extend $g$ measurably beyond $D$ with values in $[m,M]$, and take smooth convolutions $g_j$. Choose smooth boundary data $\varphi_j\to u|_{\partial D}$ uniformly. The smooth Dirichlet theory \cite{CNS} supplies admissible solutions of $\sigma_2(D^2u_j)=g_j^2$ with boundary values $\varphi_j$. Fixed quadratic and constant barriers give a common amplitude bound.
For every $D'\Subset D$, the interior estimates
\cite[(6.4) and Theorem~2.7]{TW2} for $n\geq4$ and $n=3$,
respectively, and local Lipschitz bounds for convex functions when
$n=2$, give uniform $C^{0,1/2}(D')$ bounds depending only on $n$,
the amplitude bound, $M$, and the fixed domains. In (6.4), take the
density exponent $q=n$ and $\alpha=1/2$.

The boundary modulus is also uniform. Write $D=B_R(x_D)$ and, for $y\in\partial D$, set $\ell_y(x)=R-(y-x_D)\cdot(x-x_D)/R$. This affine function is nonnegative on $\overline D$ and equals $|x-y|^2/(2R)$ on $\partial D$. Given $\varepsilon>0$, uniform convergence of the boundary data permits a constant $C_\varepsilon$, independent of large $j$ and of $y$, such that on $\partial D$
\[
u(y)-2\varepsilon-C_\varepsilon\ell_y
\leq\varphi_j\leq
u(y)+2\varepsilon+C_\varepsilon\ell_y.
\]
Add $K(|x-x_D|^2-R^2)$ to the lower barrier, where $K$ is fixed so that its Hessian satisfies $G\geq M$. The upper barrier is affine. Comparison gives a uniform boundary modulus of continuity. By weak continuity, every uniform subsequential limit has Hessian
measure $g^2\,dx$ and boundary values $u$. The comparison principle for continuous Hessian-measure solutions \cite[Theorem~3.1]{TW1} identifies this limit with $u$, proving \eqref{eqn2.9}.

Suppose now that \eqref{eqn1.4} holds. For every ball $B\Subset\Omega$,
the same smooth Dirichlet construction and compactness argument, with
constant density, give $v_c\in C^0(\overline B)\cap\Phi_2(B)$ satisfying
\[
\mu_2[v_c]=c\,dx\quad\text{in }B,\qquad v_c=u\quad\text{on }\partial B,
\qquad c\in\{\lambda,\Lambda\}.
\]
The smooth solutions converge uniformly on $\overline B$, so viscosity
stability also makes $v_c$ a viscosity solution. Viscosity comparison gives
$v_\Lambda\leq u\leq v_\lambda$.
The total-mass comparison principle \cite[Corollary~2.4]{TW1} gives
\[
\lambda|B|\leq\mu_2[u](B)\leq\Lambda|B|
\quad\text{for every ball }B\Subset\Omega.
\]
Hence $\lambda\,dx\leq\mu_2[u]\leq\Lambda\,dx$, which is equivalent to
\eqref{eqn1.3}.
\end{proof}

\subsection{An auxiliary Dirichlet problem}\label{sec-auxiliary}

We first give a trace-dependent version of the quotient calculation
in Guan and Zhang~\cite[Proposition~2.2]{GZ}.
Let $h\in C^2((0,\infty))$ be positive and satisfy
\begin{equation}\label{eqn3.3}
h^2-2shh'>0,\qquad (h-sh')^2+s^2hh''>0
\quad\text{for every }s>0.
\end{equation}
For $s=\tr H>0$ and $\phi(s)=h(s)^2/s$, set
\[
F(H)=\frac{\sigma_2(H)-h(s)^2}{s}
=\frac s2-\frac{|H|^2}{2s}-\phi(s).
\]
Direct differentiation gives
\begin{equation}\label{eqn3.4}
D^2F(H)[Z,Z]
=-\frac1s\left|Z-\frac{\tr Z}{s}H\right|^2
-\phi''(s)(\tr Z)^2,
\end{equation}
where
\[
\phi''(s)=\frac2{s^3}\bigl((h-sh')^2+s^2hh''\bigr)>0.
\]
For every unit vector $\xi$,
\begin{equation}\label{eqn3.5}
DF(H)[\xi\otimes\xi]
=\frac12|H/s-\xi\otimes\xi|^2
+\frac{h^2-2shh'}{s^2}>0.
\end{equation}
Thus $F$ is strictly concave and elliptic on $\{\tr H>0\}$.
Its zeros lie in $\Gamma_2$ and satisfy $G(H)=h(\tr H)$.

Fix $0<m\leq M$. Choose a nondecreasing
$\chi\in C^\infty(\mathbb R)$ with $0\leq\chi\leq1$,
$\chi=0$ on $(-\infty,0]$, and $\chi=1$ on $[1,\infty)$.
For a constant $\ell_\eta\geq1$ to be chosen, set
\begin{equation}\label{eqn3.1}
\eta(s)=\chi\bigl((\log s)/\ell_\eta\bigr)\quad(s>0),
\qquad\eta(0)=0,\qquad\tau=e^{\ell_\eta}.
\end{equation}
Then $\eta=0$ on $[0,1]$, $\eta=1$ on $[\tau,\infty)$, and
\begin{equation}\label{eqn3.2}
\eta(s)^2\leq\min\{s,1\},\qquad
\sup_{s>0}\bigl(|s\eta'(s)|+|s^2\eta''(s)|\bigr)
\leq C/\ell_\eta.
\end{equation}
For $g,b\in[m,M]$, $t,z\in[0,1]$, and $P>0$, define
\[
h(s)=g-tz(g-b)\eta(s/P).
\]
We have $m\leq h\leq M$ and
$|sh'|+|s^2h''|\leq C(M-m)/\ell_\eta$.
Taking $\ell_\eta=\ell_\eta(m,M)$ sufficiently large therefore gives
\[
h^2-2shh'\geq m^2/2,\qquad
(h-sh')^2+s^2hh''\geq m^2/4.
\]
Hence~\eqref{eqn3.3} holds uniformly in $g,b,t,z,P$.
We fix this choice of $\eta$ and $\tau$ throughout this section.

The following variant of the Dirichlet construction of Chen, Zhou,
and Zhu~\cite[Theorem~3.1]{CZZ} permits sign-changing $g-b$ and
uses the logarithmic cutoff~\eqref{eqn3.1}.

\begin{proposition}\label{thm3.2}
Let $D=B_R(x_D)$, let $u,g,b\in C^\infty(\overline D)$, and suppose that
\[
D^2u\in\Gamma_2,\qquad G(D^2u)=g,\qquad m\le g,b\le M
\quad\text{in }\overline D.
\]
Let $\zeta\in C_c^\infty(D)$ satisfy $0\le\zeta\le1$, and let $\eta$
be the cutoff constructed above. For every $P>0$, there is a unique
$v\in C^\infty(\overline D)$ with $D^2v\in\Gamma_2$ satisfying
\begin{equation}\label{eqn3.6}
G(D^2v)+\zeta(g-b)\eta(\Delta v/P)=g\quad\text{in }D,
\qquad v=u\quad\text{on }\partial D.
\end{equation}
Moreover,
\begin{equation}\label{eqn3.7}
\|v\|_{L^\infty(D)}\le \|u\|_{L^\infty(D)}+C R^2,
\end{equation}
where $C$ depends only on $n,m,M$.
\end{proposition}

\begin{proof}
Fix $0<c_*<m$. Since $D$ is a ball, the classical Dirichlet
theorem~\cite[Theorem~3]{CNS} gives a smooth admissible function
$w_0$ satisfying
\[
G(D^2w_0)=c_*\quad\text{in }D,\qquad w_0=u\quad\text{on }\partial D.
\]
Put
\[
q(x)=\tfrac12(|x-x_D|^2-R^2),\qquad
h_t(x,s)=g(x)-t\zeta(x)(g(x)-b(x))\eta(s/P)
\]
and
\[
F_t(x,H)=\frac{\sigma_2(H)-h_t(x,\tr H)^2}
{\tr H},\qquad 0\le t\le1.
\]
The preceding calculation makes $F_t$ smooth, concave, and elliptic on
$\{\tr H>0\}$, uniformly in $t$ on compact matrix sets, and its
zeros are admissible. We solve $F_t(x,D^2v_t)=0$ with boundary
value $u$, starting from $v_0=u$.

Choose $\kappa_0=\kappa_0(n,m,M)>0$ such that $m+\kappa_0G(I)>M$, and set $\ell=u+\kappa_0q$. The concavity and homogeneity of $G$ imply
\[
G(D^2\ell)\ge G(D^2u)+\kappa_0G(I)>M.
\]
Consequently $F_t(x,D^2\ell)>0$ and $F_t(x,D^2w_0)<0$. Comparison gives
\begin{equation}\label{eqn3.8}
\ell\le v_t\le w_0.
\end{equation}
Since $w_0$ is subharmonic, \eqref{eqn3.7} follows. In the rest of this proof, constants may depend on the fixed smooth data and $P$, but not on $t$.

Take $\kappa>\kappa_0$ and put $\underline u=u+\kappa q$. By compactness,
\[
\delta:=\min_{\overline D\times[0,1]}F_t(x,D^2\underline u)>0.
\]
For $L_t=F_t^{ij}(x,D^2v_t)\partial_{ij}$, concavity yields
\begin{equation}\label{eqn3.9}
L_t(\underline u-v_t)\ge\delta.
\end{equation}
Differentiating the equation, with $\partial_{x_k}h_t$ taken at fixed $s$, gives
\[
L_t(v_t)_k=
\frac{2h_t(x,\Delta v_t)\partial_{x_k}h_t(x,\Delta v_t)}{\Delta v_t}.
\]
Since $\Delta v_t\ge c_nm$, this expression is bounded in terms of the fixed smooth data. On $\partial D$, the comparison \eqref{eqn3.8} implies, for the outward normal $\nu$,
\[
(w_0)_\nu\le(v_t)_\nu\le\ell_\nu,
\]
while tangential derivatives are prescribed. The maximum principle applied to $\pm(v_t)_k+C(\underline u-v_t)$, followed by \eqref{eqn3.7}, therefore gives a uniform gradient bound.

We next establish the boundary Hessian bound. Choose a fixed collar of $\partial D$ on which $\zeta=0$. There $G(D^2v_t)=g$. Tangential differentiation of the boundary values controls the tangential--tangential derivatives. To bound mixed derivatives, let $T$ be an infinitesimal rotation about $x_D$, and set
\[
\varphi=T(v_t-u),\qquad W=v_t-\underline u,
\qquad L=G^{ij}(D^2v_t)\partial_{ij}.
\]
Orthogonal invariance gives $L(Tv_t)=Tg$. The functions $u$ and $v_t$ solve the same equation in the collar. Concavity therefore gives
\[
LW\le-\kappa\sum_iG^{ii},\qquad
|L\varphi|\le C\left(1+\sum_iG^{ii}\right).
\]
Newton's inequality gives $\sum_iG^{ii}\ge c_n>0$. On the outer edge of the collar, $\varphi=W=0$. On a fixed inner edge, the gradient estimate bounds $\varphi$, whereas \eqref{eqn3.8} implies
\[
W\ge(\kappa-\kappa_0)(-q)>0.
\]
For sufficiently large $C$, one has $L(\pm\varphi-CW)\geq0$ in the collar and $\pm\varphi-CW\leq0$ on its boundary. Thus $|\varphi|\le CW$ in the collar. Dividing by the inward distance at $\partial D$ bounds all tangential--normal derivatives; the derivatives of the rotation coefficients involve only the already bounded first derivatives.

For the remaining derivative, use an orthonormal tangential--normal frame and write
\[
D^2v_t=\begin{pmatrix}B_t&p_t\\p_t^T&r_t\end{pmatrix}
\quad\text{on }\partial D.
\]
The function $z=v_t-w_0$ is nonpositive and vanishes on $\partial D$, so $z_\nu\ge0$. With the second fundamental form of the sphere taken positive, tangential differentiation gives $D^2_{\tan}z=z_\nu\mathrm{II}\ge0$. Hence
\[
B_t\ge B_0:=D^2_{\tan}w_0,\qquad
\tr B_t\ge\tr B_0
=\bigl[(\Delta w_0)I-D^2w_0\bigr](\nu,\nu)\ge c_0>0.
\]
Here $c_0>0$ because $w_0$ is smooth and strictly admissible on $\overline D$. The block identity for $\sigma_2$ now gives
\[
r_t=\frac{g^2-\sigma_2(B_t)+|p_t|^2}{\tr B_t},
\]
which proves the full boundary Hessian bound.

For the global Hessian bound, fix a unit vector $\xi$ and maximize
\[
Q_\xi=(v_t)_{\xi\xi}+C(\underline u-v_t).
\]
Boundary maxima are controlled by the preceding estimates. At an interior maximum where $\Delta v_t<\tau P$, the cone inequality $|D^2v_t|\le\Delta v_t$ and $\underline u\le v_t$ bound $Q_\xi$ directly. At every other interior maximum, the cutoff equals one. Put $\widehat h_t=g-t\zeta(g-b)$. At this point, through second order,
\[
G(D^2v_t)=\widehat h_t,
\qquad L_t=\frac{2\widehat h_t}{\Delta v_t}\,
G^{ij}(D^2v_t)\partial_{ij}.
\]
These identities also hold at $\Delta v_t=\tau P$, since the cutoff is smooth and constant on $[\tau,\infty)$. Twice differentiating and using the concavity of $G$ gives
\[
L_t(v_t)_{\xi\xi}
\ge-\frac{2\widehat h_t}{\Delta v_t}\|D^2\widehat h_t\|_{L^\infty(D)}
\ge-C_1.
\]
Together with \eqref{eqn3.9}, this contradicts $L_tQ_\xi\le0$ when $C\delta>C_1$. The amplitude bound then controls all directional second derivatives from above, hence the trace and, by admissibility, the full Hessian.

The Hessian bound and $G(D^2v_t)\ge m$ confine $D^2v_t$ to a
fixed compact subset $K\Subset\Gamma_2$, where $F_t$ is uniformly
elliptic. The infimum of its supporting affine functions at matrices
in $K$ gives a concave, uniformly elliptic extension to $\Sym$
agreeing with $F_t$ on $K$;
its H\"older variation in $x$ grows at most linearly in the matrix norm.
Interior Evans--Krylov estimates~\cite{CC} and the boundary estimates
for $G(D^2v_t)=g$ in the fixed collar~\cite[Section~4]{CNS} give a
uniform $C^{2,\alpha}(\overline D)$ bound for some $\alpha>0$.
Schauder estimates give higher regularity.
For $0<\beta<\alpha$, compactness in $C^{2,\beta}(\overline D)$
proves closedness, and the implicit function theorem with linear
Dirichlet Schauder theory gives openness.

Finally, the segment joining the Hessians of two admissible solutions has positive trace. Integrating the elliptic linearization of $F_1$ along this segment and applying the maximum principle to the difference proves uniqueness.
\end{proof}

\subsection{A localized trace estimate}

The following estimate uses the maximum principle calculations of
Chou and Wang~\cite{CW}.

\begin{lemma}\label{wp-weighted-bound}
Let $V\Subset V'\Subset D$ be fixed, and let $W,v$ be smooth and
admissible in $D$. Suppose that $\mathcal O=\{\rho>0\}\Subset V$,
where $\rho=W-v-\delta_0$ for a constant $\delta_0>0$, and that,
for some $0<m\leq M$ and $A>0$,
\[
m\leq G(D^2v)\leq M,\qquad
\|W\|_{C^2(V')}+[\rho]_{C^{0,1/2}(V')}
+\|\rho_+\|_{L^\infty(V')}\leq A.
\]
Let $a\geq12$, $P\geq1$, and let $b$ be smooth with $m\leq b\leq M$.
Assume that, for every multi-index $\beta$ with $|\beta|\leq2$,
\[
D^\beta\bigl(\sigma_2(D^2v)-b^2\bigr)=0
\quad\text{on }\{\rho>0,\ \rho^a\Delta v\geq2P\}.
\]
Set
\begin{equation}\label{wp-starting-threshold}
L=\|D(b^2)\|_{L^\infty(V')},\quad
J=\|D^2(b^2)\|_{L^\infty(V')},\quad P_b=1+L^2+J.
\end{equation}
If $P\geq P_b$, then
\begin{equation}\label{wp-weighted-estimates}
\sup_{\mathcal O}\rho^{a-1}|Dv|^2\leq CP,
\qquad \sup_{\mathcal O}\rho^a\Delta v\leq CP,
\end{equation}
where $C$ depends only on $n,m,M,a,A$ and the fixed domains.
\end{lemma}

\begin{proof}
There is nothing to prove if $\mathcal O=\varnothing$.
Fix $x\in\mathcal O$ and let $r=\rho(x)$. The H\"older bound gives
a ball $B_\ell(x)\Subset V'$ such that
\begin{equation}\label{wp-holder-ball}
r/2\leq\rho\leq3r/2\quad\text{on }B_\ell(x),\qquad
\ell=\min\{\ell_*,c r^2\}.
\end{equation}
Set $w=v-W+\delta_0=-\rho$,
$\chi(y)=1-|y-x|^2/\ell^2$, and $\Phi(w)=(2r-w)^{-1/4}$.
Maximize $\chi\Phi(w)v_\xi$ over $B_\ell(x)\times\mathbb S^{n-1}$.
At a positive maximum, choose coordinates such that $Dv=pe_1$, $p>0$, and put
$Q=\chi p$. The first derivative identity is
\begin{equation}\label{wp-gradient-first}
\frac{v_{1i}}p+\frac{\Phi'}\Phi(v_i-W_i)+\frac{\chi_i}\chi=0.
\end{equation}
If $Q\leq C(\|DW\|_\infty+r/\ell)$, then
$r^{a-1}Q^2\leq C$, since $a\geq3$ and $r$ is bounded.
Otherwise~\eqref{wp-gradient-first} gives $v_{11}\leq-cp^2/r$.
Writing $N=N(D^2v)$ and $\mathcal S=\tr N$, we obtain
\begin{equation}\label{wp-gradient-cone}
N^{11}\geq\Delta v=\mathcal S/(n-1),\qquad
\mathcal S\geq cp^2/r.
\end{equation}
If $\rho^a\Delta v<2P$ at this point, then
$\Delta v\leq CP r^{-a}$ by~\eqref{wp-holder-ball}, and hence
$Q^2\leq CP r^{1-a}$.

Otherwise, differentiate $\sigma_2(D^2v)=b^2$ once
and apply $N^{ij}\partial_{ij}$ to the logarithm of the test function.
We use
\[
\frac{\Phi''}\Phi-3\left(\frac{\Phi'}\Phi\right)^2
=\frac1{8(2r-w)^2},\qquad
N^{ij}w_{ij}=2b^2-N^{ij}W_{ij}\geq-C\mathcal S,
\]
and
$N^{ij}w_iw_j\geq\frac12N^{11}p^2-C\mathcal S\|DW\|_\infty^2$.
Using~\eqref{wp-gradient-first} to bound the negative square from
$\log v_1$ gives
\[
\frac{cN^{11}p^2}{r^2}
\leq\frac Lp+C\mathcal S
\left(\frac1r+\frac{\|DW\|_\infty^2}{r^2}
+\frac1{\ell^2\chi^2}\right).
\]
Multiply by $r^2\chi^2/\mathcal S$ and
use~\eqref{wp-gradient-cone} to get
\begin{equation}\label{wp-gradient-reduced}
Q^2\leq\frac{CLr^3}{Q^3}
+C\left(r+\|DW\|_\infty^2+r^2/\ell^2\right).
\end{equation}
For $Z=r^{a-1}Q^2$, this implies
\[
Z\leq\frac{CLr^{(5a+1)/2}}{Z^{3/2}}+C
\leq\frac{CL}{Z^{3/2}}+C.
\]
Here the terms independent of $L$ are bounded by
\eqref{wp-holder-ball} and $a\geq3$. Thus $Z\leq C(1+L)\leq CP$.
The values of $\Phi$ at the maximum and at the center are comparable. This proves
the first estimate in~\eqref{wp-weighted-estimates}.

For the Hessian estimate, let
\[
H=\max_{\overline{\mathcal O}\times\mathbb S^{n-1}}
\rho^av_{\xi\xi},
\]
and choose a maximizing point $x_0$, with $r=\rho(x_0)>0$.
On $\mathcal O_r=\{\rho>r/3\}$, the gradient bound gives a number
$K_r\geq1$ such that
\begin{equation}\label{wp-local-gradient}
|Dv|+|DW|\leq K_r,\qquad K_r^2\leq CP r^{1-a}.
\end{equation}
Define
\begin{equation}\label{wp-rational-cutoff}
z(\rho)=\frac{r(3\rho-r)}{\rho+r}
=\rho-\frac{(\rho-r)^2}{\rho+r},\qquad
\phi(s)=(1-s/K_r^2)^{-1/8},
\end{equation}
and maximize
\[
\Psi=z(\rho)^a\phi(|Dv|^2/2)v_{\xi\xi}
\quad\text{on }\overline{\mathcal O_r}\times\mathbb S^{n-1}.
\]
The maximum is interior, since $z=0$ on $\partial\mathcal O_r$.
One has $1\leq\phi\leq2^{1/8}$, $0<z\leq\rho$, and $z(r)=r$.
Since
\[
\frac{z(rt)}{rt}=1-\frac{(t-1)^2}{t(t+1)}
\leq\frac{44}{45}\quad
\left(t>\frac13,\ t\notin\left(\frac34,\frac54\right)\right)
\]
and $2^{1/8}(44/45)^a<1$ for $a\geq12$, comparison with $x_0$
gives
\begin{equation}\label{wp-maximum-layer}
\frac34<\frac\rho r<\frac54
\end{equation}
at every maximum of $\Psi$. In particular,
\begin{equation}\label{wp-cutoff-derivatives}
z\asymp r,\qquad z'/z\asymp r^{-1},\qquad
E:=1-\frac{zz''}{(z')^2}=\frac{1+3\rho/r}{2}\leq\frac{19}{8}.
\end{equation}
If $\rho^a\Delta v<2P$ at the maximum of $\Psi$, then $H\leq\Psi\leq CP$.
We may therefore assume that $G(D^2v)-b$ and its derivatives
of order at most two vanish at the maximum.

Diagonalize $D^2v$ there, with $\lambda_1\geq\cdots\geq\lambda_n$,
and take the maximizing direction to be $e_1$. Put
\[
F^{ii}=G^{ii}(D^2v)=\frac{\Delta v-\lambda_i}{2b},\qquad
\mathcal F=\sum_iF^{ii},\qquad \theta=\phi'/\phi.
\]
The first derivative identity is
\begin{equation}\label{wp-hessian-first}
0=a\frac{z'}z\rho_i+\theta v_i\lambda_i+
\frac{v_{11i}}{\lambda_1}.
\end{equation}
Differentiating twice and contracting with $F^{ij}$ gives
\begin{align}\label{wp-hessian-second}
0\geq{}&a\frac{z'}z\sum_iF^{ii}\rho_{ii}
-aE\left(\frac{z'}z\right)^2\sum_iF^{ii}\rho_i^2\notag\\
&+\theta\left(\sum_iF^{ii}\lambda_i^2+Dv\cdot Db\right)
+\left(\frac{\phi''}\phi-\theta^2\right)
\sum_iF^{ii}(v_i\lambda_i)^2\notag\\
&+\frac{b_{11}-G^{ij,rs}v_{ij1}v_{rs1}}{\lambda_1}
-\sum_iF^{ii}\frac{v_{11i}^2}{\lambda_1^2}.
\end{align}
Admissibility of $W$ gives $\sum_iF^{ii}\rho_{ii}\geq-b$.
Also $|D\rho|\leq K_r$, $|Db|\leq CL$,
$|D^2b|\leq C(J+L^2)$, and
\begin{equation}\label{wp-phi-identities}
\theta\asymp K_r^{-2},\qquad \phi''/\phi=9\theta^2.
\end{equation}

Fix $\varepsilon_n>0$ sufficiently small, and put
$\mathcal A=\mathcal F$ if $\lambda_2\geq\varepsilon_n\lambda_1$
and $\mathcal A=F^{11}$ otherwise. We claim that
\begin{equation}\label{wp-nonseparated}
0\geq cK_r^{-2}\mathcal A\lambda_1^2
-C\mathcal A K_r^2r^{-2}
-C\left(r^{-1}+L/K_r+(J+L^2)/\lambda_1\right).
\end{equation}
In the first case, writing $T=\Delta v$, we have
$T>|D^2v|\geq\sqrt2\lambda_2$, whence
$F^{22}\geq c\mathcal F$ and
$\sum_iF^{ii}\lambda_i^2\geq c\mathcal F\lambda_1^2$.
Substitute~\eqref{wp-hessian-first} in the last square of
\eqref{wp-hessian-second}, and use concavity and
\eqref{wp-cutoff-derivatives}--\eqref{wp-phi-identities}.

In the second case, make this substitution only for $i=1$.
For $i\geq2$, substitute~\eqref{wp-hessian-first} into the
negative cutoff square instead. Together with the last square of
\eqref{wp-hessian-second}, these terms are bounded below by
\[
-\left(1+\frac{2E}{a}\right)F^{ii}\frac{v_{11i}^2}{\lambda_1^2}
-\frac{2E}{a}\theta^2F^{ii}(v_i\lambda_i)^2.
\]
Here $1+2E/a\leq67/48<3/2$, and the second term is absorbed by
\eqref{wp-phi-identities}. The off-diagonal second derivatives of
$G$ and concavity on diagonal matrices give
\begin{align*}
-\frac{G^{ij,rs}v_{ij1}v_{rs1}}{\lambda_1}
-\frac32\sum_{i\geq2}F^{ii}\frac{v_{11i}^2}{\lambda_1^2}
&\geq\sum_{i\geq2}
\frac{\lambda_1-\frac34(\Delta v-\lambda_i)}{b\lambda_1^2}
v_{11i}^2\geq0,
\end{align*}
since $\Delta v-\lambda_i\leq[1+(n-2)\varepsilon_n]\lambda_1$
and we choose $(n-2)\varepsilon_n\leq1/3$ (with no restriction
needed for $n=2$). This proves~\eqref{wp-nonseparated}.

If the second term in~\eqref{wp-nonseparated} cannot be absorbed
by half the first, then $r\lambda_1\leq CK_r^2$.
Otherwise use $\mathcal A\lambda_1\geq F^{11}\lambda_1\geq b/n$.
Indeed, for $s=\sum_{i\geq2}\lambda_i>0$, Newton's inequality and
$s\leq(n-1)\lambda_1$ give
\[
b^2=\lambda_1s+\sigma_2(\lambda_2,\ldots,\lambda_n)
\leq\lambda_1s+\frac{n-2}{2(n-1)}s^2
\leq\frac n2\lambda_1s.
\]
It follows that
\[
\lambda_1^2\leq C(K_r^2/r+K_rL)\lambda_1
+CK_r^2(J+L^2).
\]
The quadratic inequality and the preceding alternatives give
\begin{equation}\label{wp-local-hessian}
\lambda_1\leq C\left(1+K_r^2/r+K_rL+K_r\sqrt J\right).
\end{equation}
By~\eqref{wp-maximum-layer}, $H\leq Cr^a\lambda_1$ at this point.
Insert~\eqref{wp-local-gradient} into~\eqref{wp-local-hessian} and
use $L,\sqrt J\leq\sqrt P$ to obtain
\[
H\leq C\left(r^a+P+Pr^{(a+1)/2}\right)\leq CP.
\]
Finally $\Delta v\leq n\lambda_{\max}(D^2v)$, which proves
the second estimate in~\eqref{wp-weighted-estimates}.
\end{proof}

\begin{corollary}\label{thm3.3}
Let $0<m\leq M$, $A>0$, and $\tau\geq1$. Suppose that
$\eta\in C^\infty([0,\infty))$, $0\leq\eta\leq1$, and
$\eta=1$ on $[\tau,\infty)$. Let $g,b\in C^\infty(B_2)$ satisfy
$m\leq g,b\leq M$, and suppose that $v\in C^\infty(B_2)$ is an
admissible solution of
\begin{equation}\label{eqn3.10}
G(D^2v)+(g-b)\eta(\Delta v/P)=g\quad\text{in }B_2,
\end{equation}
with $\|v\|_{L^\infty(B_2)}\leq A$ and $P\geq1$. Set
\[
L=\|D(b^2)\|_{L^\infty(B_2)},\qquad
J=\|D^2(b^2)\|_{L^\infty(B_2)}.
\]
Then
\begin{equation}\label{eqn3.11}
\sup_{B_{1/2}}\Delta v\leq C_0P+C(1+L^2+J),
\end{equation}
where $C_0,C$ depend only on $n,m,M,A,\tau$.
\end{corollary}

\begin{proof}
Equation~\eqref{eqn3.10} gives $m\leq G(D^2v)\leq M$.
Let $v_0$ solve
\[
G(D^2v_0)=m/2\quad\text{in }B_{3/2},\qquad
v_0=v\quad\text{on }\partial B_{3/2},
\]
using the smooth admissible Dirichlet theory~\cite{CNS}.
If $h$ is the harmonic extension of the boundary values, comparison
gives $v\leq v_0\leq h$, hence $\|v_0\|_\infty\leq A$.
The constant-density estimate~\cite[Theorem~1.1]{LW}, applied to
$2v_0/m$, bounds $D^2v_0$ on fixed smaller balls.

Put $q=v_0-v$ and $N_0=N(D^2v_0)$. Concavity gives
\[
N_0^{ij}q_{ij}
=m\,DG(D^2v_0):D^2q\leq-m^2/2.
\]
On $B_1$, $N_0$ is positive definite and $\tr N_0\leq C$.
Comparison on $B_{1/4}(x)$ with a sufficiently small fixed multiple of
$1/16-|\,\cdot-x|^2$, for $x\in B_{1/2}$, yields
\[
q\geq c_*>0\quad\text{on }B_{1/2},
\]
where $c_*$ depends only on $n,m,M,A$.
The local H\"older estimate used in Lemma~\ref{thm2.4} gives a uniform
$C^{0,1/2}(B_{7/4})$ bound for $v$. The Poisson formula consequently gives
\[
0\leq q\leq h-v
\leq C\dist(x,\partial B_{3/2})^{1/2}
\quad\text{in }B_{3/2}.
\]
With $\delta=c_*/4$ and $\rho=v_0-v-\delta$, choose fixed balls
$V\Subset V'\Subset B_{3/2}$ so that
\[
\overline{\{\rho>0\}}\subset V,\qquad
\rho\geq3c_*/4\quad\text{on }B_{1/2}.
\]
The preceding estimates also give
\[
\|v_0\|_{C^2(V')}+[\rho]_{C^{0,1/2}(V')}
+\|\rho_+\|_{L^\infty(V')}\leq C,
\]
with all choices depending only on $n,m,M,A$.

Apply Lemma~\ref{wp-weighted-bound} with $W=v_0$, $\delta_0=\delta$,
$a=12$, and threshold
\[
\widetilde P=\max\left\{1+L^2+J,\frac\tau2R_*^aP\right\},
\qquad R_*=\max\{1,\|\rho_+\|_{L^\infty(V')}\}.
\]
Indeed, $\rho^a\Delta v\geq2\widetilde P$ implies
$\Delta v\geq\tau P$. At such a point~\eqref{eqn3.10} agrees through
second derivatives with $G(D^2v)=b$, since $\eta=1$ and
$\eta'=\eta''=0$ on $[\tau,\infty)$. The lemma gives
\[
\sup_{\{\rho>0\}}\rho^a\Delta v\leq C\widetilde P.
\]
The lower bound for $\rho$ on $B_{1/2}$ proves~\eqref{eqn3.11}.
\end{proof}

By rescaling and a finite covering, Corollary~\ref{thm3.3} applies to
any fixed pair of relatively compact nested balls, with the constants also
depending on the two radii. The rescaling
$v_r(y)=r^{-2}v(x_0+ry)$ leaves both the Hessian and the threshold $P$
unchanged.

\subsection{Absolute continuity of the Hessian}\label{sec-absolute}

\begin{proof}[Proof of Theorem~\ref{thm1.1}]
Fix concentric balls
\[
U\Subset V_1\Subset V_2\Subset D\Subset\Omega
\]
and choose $\zeta\in C_c^\infty(D)$, $0\leq\zeta\leq1$, equal to one near $\overline V_2$. Write
\begin{equation}\label{eqn4.1}
\mu_2[u]=g^2\,dx,\qquad
D^2u=A\,dx+(D^2u)_{\sing},\qquad
\Delta u=T\,dx+\mu_s,
\end{equation}
where $m\leq g\leq M$ and $T=\tr A\in L^1(D)$. The trace measure is finite on $D$ because $D\Subset\Omega$.

Fix an arbitrary $b\in C^\infty(\overline D)$ with $m\leq b\leq M$, and let $P\geq1$. Take the approximation $u_j,g_j$ of Lemma~\ref{thm2.4}, and use Proposition~\ref{thm3.2} to solve
\begin{equation}\label{eqn4.2}
G(D^2v_j)+\zeta(g_j-b)\eta(\Delta v_j/P)=g_j
\quad\text{in }D,\qquad v_j=u_j\quad\text{on }\partial D.
\end{equation}
The solutions have a common amplitude bound, and
\begin{equation}\label{eqn4.3}
m\leq h_j:=G(D^2v_j)\leq M,
\qquad\sup_{V_1}\Delta v_j\leq C_0P+C_b
\end{equation}
by Corollary~\ref{thm3.3}. The constants $C_0$ and the amplitude bound depend only on the fixed domains, $n,m,M$, and $\|u\|_{L^\infty(D)}$. In particular, $C_0$ is independent of $b$ and its derivatives.

Put
\[
\theta_j=\zeta^2\eta(\Delta v_j/P)^2.
\]
The local Sobolev bounds in \cite[(4.2)]{TW2}, Rellich compactness,
and weak-star compactness in $L^\infty(D)$ give, along a common
subsequence,
\begin{equation}\label{eqn4.4}
v_j\to v\text{ in }L^1_{\mathrm{loc}}(D),\qquad
h_j\rightharpoonup^*\bar h,\qquad
\theta_j\rightharpoonup^*\theta.
\end{equation}
Here $v\in\Phi_2(D)$ is locally bounded, $m\leq\bar h\leq M$, and $0\leq\theta\leq1$. Let $B=d(D^2v)_{\ac}/dx$ and $S=\tr B$. The pointwise inequality $\theta_j\leq\Delta v_j/P$ passes to measures. Taking its absolutely continuous part yields
\begin{equation}\label{eqn4.5}
0\leq\theta\leq\min\{1,S/P\}\quad\text{almost everywhere in }D.
\end{equation}

Consider the nonnegative measures
\[
\nu_j=\mathcal B[u_j,v_j]-2g_jh_j\,dx
=\mathscr D(D^2u_j,D^2v_j)\,dx.
\]
The comparison identity~\eqref{eqn2.8} gives
\begin{equation}\label{eqn4.6}
\nu_j(D)\leq\int_D|g_j-b|^2\theta_j\,dx.
\end{equation}
By mixed-measure continuity and \eqref{eqn4.4}, these measures converge locally weakly to
\begin{equation}\label{eqn4.7}
\nu=\mathcal B[u,v]-2g\bar h\,dx\geq0.
\end{equation}
Indeed, $g_j\to g$ in $L^1$ and $h_j$ are uniformly bounded. The right-hand side of \eqref{eqn4.6} converges to $\int_D|g-b|^2\theta$, because $|g_j-b|^2\to|g-b|^2$ in $L^1(D)$. Lower semicontinuity of the mass gives
\begin{equation}\label{eqn4.8}
\nu(D)\leq\int_D|g-b|^2\theta\,dx<\infty.
\end{equation}

Write $d=d\nu_{\ac}/dx\geq0$. Lemma~\ref{thm2.1} gives
$G(A)\geq m$, so $T>0$ almost everywhere. Since $|B|\leq S$,
\[
\mathscr B(A,B)\geq S(T-|A|)
=\frac{2G(A)^2S}{T+|A|}\geq m^2\frac ST.
\]
Together with~\eqref{eqn2.5}, this yields
\begin{equation}\label{eqn4.9}
d\geq\mathscr B(A,B)-2g\bar h
\geq m^2\frac ST-2M^2.
\end{equation}

Set $e=g-b$ and $R=4M^2/m^2$. On $\{S>RT\}$, \eqref{eqn4.9} gives $d\geq2M^2$, and hence $e^2\theta\leq d/2$. On $\{S\leq RT\}$, \eqref{eqn4.5} gives $e^2\theta\leq Re^2T/P$. Substitution in \eqref{eqn4.8} yields
\[
\nu(D)\leq\frac RP\int_D|g-b|^2T\,dx+\frac12\nu_{\ac}(D).
\]
As $\nu_{\ac}(D)\leq\nu(D)$, we have proved
\begin{equation}\label{eqn4.10}
P\nu(D)\leq2R\int_D|g-b|^2T\,dx.
\end{equation}

We now choose $P\geq\max\{1,C_b\}$. Then \eqref{eqn4.3} gives $\Delta v_j\leq C_1P$ on $V_1$, with $C_1=C_0+1$ independent of $b$. Let $\kappa=c/P$, where $c=c(n,m,C_1)>0$ is sufficiently small. Since
\[
\sigma_2(H-\kappa I)
=\sigma_2(H)-(n-1)\kappa\tr H+\binom n2\kappa^2,
\]
we have $\sigma_2(D^2v_j-\kappa I)\geq m^2/2$ on $V_1$. Decreasing $c$ if necessary also gives $\tr(D^2v_j-\kappa I)>0$: use $\Delta v_j\geq\sqrt{2n/(n-1)}\,m$ and $P\geq1$. Thus
\[
v_j-\kappa|x|^2/2\in\Phi_2(V_1).
\]
Passing to the local $L^1$ limit gives $v-\kappa|x|^2/2\in\Phi_2(V_1)$. Positivity and bilinearity of the mixed measure, together with \eqref{eqn2.3}, imply
\[
\mathcal B[u,v]\geq(n-1)\kappa\Delta u\quad\text{on }V_1.
\]
The term subtracted in \eqref{eqn4.7} is absolutely continuous. Taking singular parts and then restricting to $U$ gives
\begin{equation}\label{eqn4.11}
\nu_s(U)\geq\frac{c(n-1)}P\mu_s(U).
\end{equation}
Combining \eqref{eqn4.10} and \eqref{eqn4.11} proves
\begin{equation}\label{eqn4.12}
\mu_s(U)\leq C\int_D|g-b|^2T\,dx,
\end{equation}
where $C$ is independent of $b$ and its derivatives.

Finally, extend $g$ with values in $[m,M]$ and let $b_\ell$ be its
smooth convolutions. Then $b_\ell\to g$ almost everywhere and
$m\leq b_\ell\leq M$. Since $T\in L^1(D)$, dominated convergence
in~\eqref{eqn4.12} gives
$\mu_s(U)=0$. The domination $|D^2u|\leq\Delta u$ then eliminates the
singular part of $D^2u$ on $U$. As the balls were arbitrary,
$u\in W^{2,1}_{\mathrm{loc}}(\Omega)$.

To prove~\eqref{w21-interior-estimate}, choose
$\chi\in C_c^\infty(B_{2r}(x_0))$ with $0\leq\chi\leq1$,
$\chi=1$ on $B_r(x_0)$, and $|D^2\chi|\leq C_nr^{-2}$.
Set $c=\inf_{B_{2r}(x_0)}u$. Integration by parts gives
\begin{align*}
\int_{B_r(x_0)}|D^2u|\,dx
&\leq\int_{B_{2r}(x_0)}\chi\,\Delta u\,dx\\
&=\int_{B_{2r}(x_0)}(u-c)\Delta\chi\,dx
\leq C_nr^{n-2}\osc_{B_{2r}(x_0)}u.\qedhere
\end{align*}
\end{proof}

\section{Interior \texorpdfstring{$W^{2,p}$}{W2,p} regularity}\label{sec-w2p}

\subsection{Constant-density comparison}

For a ball $B_{2r}(x_0)\Subset\Omega$, set
\[
u_r(y)=r^{-2}u(x_0+ry),\qquad f_r(y)=f(x_0+ry).
\]
Since $D^2u_r(y)=D^2u(x_0+ry)$, it suffices to prove
\eqref{wp-interior-estimate} on $B_1$ with $B_2\Subset\Omega$.
We use $u,f$ for the rescaled functions and,
subtracting a constant, assume
$0\leq u\leq A_*:=\osc_{B_2}u$ on $B_2$. Set
\[
D=B_{3/2},\qquad U_1=B_{9/8},\qquad
g=\sqrt f,
\]
and fix
\[
c_*=\sqrt\lambda/8,\qquad m=\sqrt\lambda/4,\qquad
M=\sqrt\Lambda+1.
\]
In this subsection, constants depend only on $n$, $\lambda$,
$\Lambda$, and $A_*$, unless otherwise indicated.

The estimates in the proof of Lemma~\ref{thm2.4} give
$[u]_{C^{0,1/2}(B_{7/4})}\leq C$.

Let $W$ be the continuous admissible viscosity solution of
\[
G(D^2W)=c_*\quad\text{in }D,\qquad W=u\quad\text{on }\partial D.
\]
Let $H$ be the harmonic extension of $u|_{\partial D}$. Comparison gives
$0\leq u\leq W\leq H\leq A_*$.
The constant-density estimates~\cite[Theorems~1.1--1.2]{LW}, applied
to $W/c_*$ on interior balls, imply that $W$ is smooth in $D$ and
$K_W:=\sup_{B_{5/4}}\Delta W\leq C$.
Fix $r_*=1/16$, and choose $0<t\leq1$, depending only on these data,
such that
\[
c_*^2+(n-1)tK_W+\binom n2t^2<\lambda.
\]
For $x\in\overline{U_1}$, the function
$W_x(y)=W(y)+\tfrac t2(|y-x|^2-r_*^2)$ is admissible and satisfies
$\sigma_2(D^2W_x)<\lambda\leq f$ in $B_{r_*}(x)$, with
$W_x=W\geq u$ on its boundary. Viscosity comparison at the center
therefore gives, with $\delta_0=tr_*^2/8$,
\begin{equation}\label{wp-fixed-gap}
W-u\geq4\delta_0\quad\text{on }\overline{U_1}.
\end{equation}

The Poisson formula and the H\"older bound for $u$ give
\begin{equation}\label{wp-boundary-gap}
0\leq W(x)-u(x)\leq H(x)-u(x)
\leq C\dist(x,\partial D)^{1/2}\quad(x\in D).
\end{equation}
Choose $0<\ell\leq1/8$ such that $C\sqrt\ell\leq\delta_0/4$, and set
\[
V=B_{3/2-\ell/2},\qquad V'=B_{3/2-\ell/4}.
\]
Use the construction in the proof of Lemma~\ref{thm2.4}, choosing
$g_j$ by convolution of the continuous function $g$ on a neighborhood
of $\overline D$ and common smooth boundary values
$\varphi_j\to u|_{\partial D}$ uniformly. It gives smooth admissible
Dirichlet solutions $u_j,W_j$ with respective $G$-densities $g_j,c_*$
and boundary values $\varphi_j$, such that
\[
g_j\to g\quad\text{uniformly on }\overline D,\qquad
u_j\to u,\quad W_j\to W\quad\text{uniformly on }\overline D.
\]
Discard finitely many terms so that
$\|u_j-u\|_\infty+\|W_j-W\|_\infty\leq\min\{1,\delta_0/4\}$.
Then~\eqref{wp-fixed-gap} and~\eqref{wp-boundary-gap} give
\begin{equation}\label{wp-uniform-gap}
\rho_{u,j}:=W_j-u_j-\delta_0\geq2\delta_0\quad\text{on }U_1,
\qquad \{\rho_{u,j}>0\}\Subset V.
\end{equation}
Indeed, $\rho_{u,j}<0$ whenever $\dist(x,\partial D)\leq\ell$.
The functions $W_j$ have a uniform $C^2(V')$ bound by the same constant-density estimate.
For large $j$, the densities $g_j$ lie in $[m,M]$. All these bounds
depend only on $n$, $\lambda$, $\Lambda$, and $A_*$.

Fix
\begin{equation}\label{wp-exponent}
a=\max\{12,n\}.
\end{equation}
Let $0<d\leq\min\{1,m\}$, with further restrictions specified below, and set
\[
\omega_f(t)=\sup_{\substack{x,y\in B_2\\|x-y|\leq t}}|f(x)-f(y)|.
\]
Choose $0<\varepsilon<1/4$ such that
\[
\omega_f(\varepsilon)\leq\sqrt\lambda\,d/4.
\]
For a standard mollifier $\vartheta_\varepsilon$, define
$b=g*\vartheta_\varepsilon-d/2$ on $\overline D$; the convolution uses
only values of $g$ in $B_{7/4}$. Since
$|\sqrt s-\sqrt t|\leq|s-t|/(2\sqrt\lambda)$ for $s,t\geq\lambda$,
we have $\|g*\vartheta_\varepsilon-g\|_{L^\infty(D)}\leq d/8$.
For all $j$ with $\|g_j-g\|_\infty\leq d/8$, this gives
\begin{equation}\label{wp-minorant}
m\leq b\leq g_j\leq M,\qquad 0\leq g_j-b\leq d.
\end{equation}
In fact, $d/4\leq g_j-b\leq3d/4$ and
$b\geq\sqrt\lambda-5d/8>m$.
Differentiating the convolution gives
\begin{equation}\label{wp-modulus-threshold}
P_b=1+\|D(b^2)\|_{L^\infty(V')}^2
+\|D^2(b^2)\|_{L^\infty(V')}
\leq C(n,\Lambda)(1+\varepsilon^{-2}).
\end{equation}

\subsection{The auxiliary equation}

Fix a nondecreasing $\eta\in C^\infty(\mathbb R)$ such that
$0\leq\eta\leq1$, $\eta=0$ on $(-\infty,1]$, and
$\eta=1$ on $[2,\infty)$. We retain the ball
$D$, the constant $\delta_0>0$, and the exponent $a$ in
\eqref{wp-exponent}, and omit the approximation index.

\begin{proposition}\label{wp-solvability}
Let $u,W,g,b\in C^\infty(\overline D)$ satisfy
\begin{gather*}
D^2u,D^2W\in\Gamma_2,\qquad
u=W\quad\text{on }\partial D,\\
G(D^2u)=g,\qquad G(D^2W)=c_*<m,\\
m\leq b\leq g\leq M,\qquad 0\leq g-b\leq d.
\end{gather*}
There is $d_*=d_*(m,M,\eta)\in(0,m]$ such that, for $d\leq d_*$ and
$P\geq1$, the Dirichlet problem
\begin{equation}\label{wp-auxiliary}
\begin{cases}
G(D^2v)+(g-b)\eta\bigl((W-v-\delta_0)_+^a\Delta v/P\bigr)=g
&\text{in }D,\\
v=u&\text{on }\partial D
\end{cases}
\end{equation}
has a unique solution $v\in C^\infty(\overline D)$ with
$D^2v\in\Gamma_2$. Moreover,
\begin{equation}\label{wp-order}
u\leq v\leq W,\qquad m\leq G(D^2v)\leq M.
\end{equation}
\end{proposition}

\begin{proof}
For $0\leq t\leq1$, put
\[
h_t(x,z,S)=g(x)-t(g(x)-b(x))
\eta\bigl((W(x)-z-\delta_0)_+^aS/P\bigr)
\]
and, for $S=\tr H>0$, define
\begin{equation}\label{wp-quotient}
F_t(x,z,H)=\frac{\sigma_2(H)-h_t(x,z,S)^2}{S}.
\end{equation}
We have $m\leq h_t\leq M$, $(h_t)_z\geq0$, and $(h_t)_S\leq0$.
For fixed $x,z,t$, write $h=h_t$ and
$s=(W-z-\delta_0)_+^aS/P$. Then
\[
h^2-2Shh_S\geq m^2
\]
and
\begin{align*}
(h-Sh_S)^2+S^2hh_{SS}
&=\bigl(h+t(g-b)s\eta'(s)\bigr)^2
-ht(g-b)s^2\eta''(s)\\
&\geq m^2-CMd.
\end{align*}
Choose $d_*>0$ so that the last expression is positive whenever
$d\leq d_*$. By~\eqref{eqn3.3}--\eqref{eqn3.5}, $F_t$ is concave and
elliptic in $H$, and $(F_t)_z\leq0$. The functions $F_t$ are smooth:
near $W-z-\delta_0=0$ with finite $S$, the cutoff is constant.

We use the method of continuity for
$F_t(x,v_t,D^2v_t)=0$ with boundary value $u$, starting from $v_0=u$.
Every solution with positive trace is admissible, and the subsolution
$u$ and supersolution $W$ give
$u\leq v_t\leq W$. Hence there is a fixed collar of $\partial D$
on which $W-v_t-\delta_0<0$ and $G(D^2v_t)=g$.
In the rest of this proof, constants may depend on the fixed smooth data,
$\delta_0$, and $P$, but not on $t$.

Set $q(x)=(|x|^2-9/4)/2$ and choose $\kappa>0$ so that
$\underline u=u+\kappa q$ satisfies $G(D^2\underline u)>M$.
Since $h_t\leq M$, there is $c>0$ such that
$F_t(x,z,D^2\underline u)\geq c$ uniformly in $x,z,t$.
Concavity in $H$, $\underline u\leq v_t$, and $(F_t)_z\leq0$ give
\begin{equation}\label{wp-strict-barrier}
\mathcal L(\underline u-v_t)\geq c,
\qquad
\mathcal L=F_t^{ij}\partial_{ij}+(F_t)_z.
\end{equation}
The quantities $(F_t)_{x_k}(x,v_t,D^2v_t)$ are bounded independently
of $D^2v_t$. Indeed, $S\geq c_nm$, and the term in $(h_t)_{x_k}$
containing $S$ is
\[
-\frac{ta(g-b)}P\eta'(s)
(W-v_t-\delta_0)_+^{a-1}W_kS;
\]
the factor $S$ cancels with the denominator in~\eqref{wp-quotient}.
Differentiation gives $\mathcal L(v_t)_k=-(F_t)_{x_k}$.
On $\partial D$, the comparison bounds yield
$W_\nu\leq(v_t)_\nu\leq u_\nu$, and tangential derivatives are
prescribed. Apply the maximum principle to
$\pm(v_t)_k+C(\underline u-v_t)$, using~\eqref{wp-strict-barrier},
to obtain a global gradient bound.

The boundary Hessian argument of Proposition~\ref{thm3.2} applies
with $W$ in place of $w_0$. Indeed, $G(D^2v_t)=g$ in the fixed
collar and $u\leq v_t\leq W$; in particular,
$v_t-\underline u\geq-\kappa q$ supplies the inner-edge barrier,
and strict admissibility of $W$ supplies the positive tangential trace
used to recover the pure normal second derivative.

It remains to bound the trace in the interior. Write $T=\Delta v_t$
and $B(x,z,T)=h_t(x,z,T)^2$. Then $B_z\geq0$ and $B_T\leq0$.
In the transition region,
$(W-v_t-\delta_0)^aT/P\in[1,2]$; partial differentiation at fixed
$z,T$ gives
\begin{gather*}
|B_{xx}|+|B_{xz}|+|B_{zz}|\leq C(1+T^{2/a}),\\
|B_{xT}|+|B_{zT}|\leq C(1+T^{1/a-1}),\qquad
|B_{TT}|\leq C/T^2.
\end{gather*}
These bounds hold also where the cutoff is constant. Taking the
Laplacian of $\sigma_2(D^2v_t)=B(x,v_t,T)$ yields
\begin{align}\label{wp-trace-identity}
\bigl(N(D^2v_t)-B_TI\bigr):D^2T
={}&B_zT+B_{zz}|Dv_t|^2+2B_{xz}\cdot Dv_t
+\tr B_{xx}\notag\\
&+2B_{zT}Dv_t\cdot DT+2B_{xT}\cdot DT
+B_{TT}|DT|^2\notag\\
&+|D^3v_t|^2-|DT|^2.
\end{align}
At an interior maximum of $T+|x|^2$, the left-hand side is at most
$-2(n-1)T$, whereas the right-hand side is bounded below by
$-C(1+T^{2/a})$. Since $a>2$, this bounds $T$. Boundary maxima
are already controlled, and $|D^2v_t|\leq T$ gives the full Hessian
bound.

The Hessians now remain in a compact subset of $\Gamma_2$,
so the operators are uniformly elliptic. The gradient bound makes
$F_t(x,v_t(x),H)$ uniformly H\"older in $x$ on this compact set.
The regularity and continuity arguments of
Proposition~\ref{thm3.2} therefore apply, using the full linearization
$\mathcal L$ for openness; its zeroth-order coefficient is
nonpositive. The maximum principle for the linearization between
two solutions also proves uniqueness.
\end{proof}

\subsection{Contraction of the trace tails}\label{sec-tails}

For $A,B\in\Gamma_2$, set
$T=\tr A$, $S=\tr B$, $g=G(A)$, and $h=G(B)$. Then
\begin{equation}\label{wp-square-defect}
\mathscr D(A,B)
=\frac{TS}{2}\left|\frac AT-\frac BS\right|^2
+\left(g\sqrt{\frac ST}-h\sqrt{\frac TS}\right)^2.
\end{equation}
This follows by expanding the squares and using
$|A|^2=T^2-2g^2$ and $|B|^2=S^2-2h^2$.
When $m\leq g,h\leq M$, it implies
\begin{equation}\label{wp-ratio-defect}
\mathscr D(A,B)\geq m^2(T/S+S/T)-2M^2.
\end{equation}

\begin{proof}[Proof of Theorem~\ref{thm-w2p}]
Use the domains and smooth approximations fixed above, satisfying
\eqref{wp-uniform-gap}. The uniform amplitude bound and the cutoff
argument proving~\eqref{w21-interior-estimate} give
\begin{equation}\label{wp-trace-mass}
\sup_j\int_V\Delta u_j\leq M_T<\infty.
\end{equation}
Until $p$ is fixed below, all constants except $P_b$ depend only on
$n$, $\lambda$, $\Lambda$, and $A_*$. In particular, they are independent
of $b$ and its derivatives, of $P$, and of the approximation index.

Let $b$ satisfy~\eqref{wp-minorant}, with $d\leq\min\{d_*,1,m\}$.
For every large $j$ and every $P\geq P_b$, let $v=v_{j,P}$ solve
\eqref{wp-auxiliary} with $u=u_j$, $W=W_j$, and $g=g_j$.
Suppress the index $j$ and write
\[
T=\Delta u,\quad S=\Delta v,\quad h=G(D^2v),\quad
\rho_u=W-u-\delta_0,\quad \rho=W-v-\delta_0,
\]
\begin{equation}\label{wp-weighted-traces}
X=(\rho_u)_+^aT,\qquad Y=\rho_+^aS,
\qquad \mathcal A_P=\{\eta(Y/P)>0\}.
\end{equation}
The order~\eqref{wp-order} gives $\rho\leq\rho_u$ and
$\mathcal A_P\subset\{Y>P\}\subset V$.
The functions $v$ have uniform amplitude and density bounds, hence
a uniform $C^{0,1/2}(V')$ bound, and
$\{\rho>0\}\subset\{\rho_u>0\}\Subset V$.
On $\{\rho>0,\ Y\geq2P\}$, equation~\eqref{wp-auxiliary} agrees
with $\sigma_2(D^2v)=b^2$ through second order, including at $Y=2P$,
since $\eta$ is smooth and constant on $[2,\infty)$.
Lemma~\ref{wp-weighted-bound} therefore gives
\begin{equation}\label{wp-comparison-cap}
Y\leq C_1P,
\end{equation}
with a uniform constant $C_1\geq1$.

Put $\mathscr D=\mathscr D(D^2u,D^2v)$. By~\eqref{eqn2.8},
\begin{equation}\label{wp-energy}
\int_D\mathscr D\leq\int_D(g-h)^2\leq d^2|\mathcal A_P|.
\end{equation}
Fix
\begin{equation}\label{wp-tail-constants}
c=\frac{m^2}{4M^2},\qquad
K=\frac{2^{a+2}C_1M^2}{m^2},\qquad
\mu(P)=|\{X>cP\}|.
\end{equation}
On $\mathcal A_P\cap\{X\leq cP\}$ one has
$S/T\geq Y/X>1/c$, and hence $\mathscr D\geq2M^2$ by
\eqref{wp-ratio-defect}. It follows that
\[
|\mathcal A_P|\leq\mu(P)+\frac{d^2}{2M^2}|\mathcal A_P|.
\]
Since $d\leq M$, we conclude
\begin{equation}\label{wp-active-volume}
|\mathcal A_P|\leq2\mu(P),\qquad
\int_D\mathscr D\leq2d^2\mu(P).
\end{equation}

Set $\mathcal N=N((D^2u+D^2v)/2)$. The midpoint linearization gives
\[
\mathcal N:D^2(v-u)=h^2-g^2,
\qquad v-u=0\quad\text{on }\partial D.
\]
For $H\in\Gamma_2$, let $\lambda_i$ be its eigenvalues and
$\nu_i=\tr H-\lambda_i$ those of $N(H)$.
Then $\nu_i\nu_j\geq\sigma_2(H)$ for $i\ne j$. Indeed, writing
$r=\sum_{k\ne i,j}\lambda_k$, the difference $\nu_i\nu_j-\sigma_2(H)$
is $\frac12(r^2+\sum_{k\ne i,j}\lambda_k^2)\geq0$.
Multiplication over all pairs gives
$\det N(H)\geq\sigma_2(H)^{n/2}$. Concavity of $G$ at the midpoint
therefore gives $\det\mathcal N\geq m^n$.
Since $|h^2-g^2|\leq2Md\,\mathbf1_{\mathcal A_P}$,
the ABP estimate in determinant form~\cite[Lemma~9.3]{GT} yields
\begin{equation}\label{wp-abp}
\begin{split}
e_P:=\|v-u\|_{L^\infty(D)}
&\leq C_n\operatorname{diam}(D)
\left\|\frac{g^2-h^2}{(\det\mathcal N)^{1/n}}\right\|_{L^n(D)}\\
&\leq Cd|\mathcal A_P|^{1/n}
\leq Cd\mu(P)^{1/n}.
\end{split}
\end{equation}

On $\{\rho_u>2e_P\}$, one has $\rho\geq\rho_u/2$.
Consequently~\eqref{wp-comparison-cap} implies
\[
\frac TS\geq\frac{X}{2^aC_1P}.
\]
On $\{X>KP,\ \rho_u>2e_P\}$, the choice of $K$ and
\eqref{wp-ratio-defect} give
\[
X\leq(2^{a+1}C_1/m^2)P\mathscr D.
\]
Together with the energy bound and~\eqref{wp-trace-mass}, this gives
\begin{align*}
\int_{\{X>KP,\ \rho_u>2e_P\}}X&\leq Cd^2P\mu(P),\\
\int_{\{0<\rho_u\leq2e_P\}}X
&\leq(2e_P)^aM_T\leq Cd^a\mu(P)^{a/n}.
\end{align*}
Since $a\geq\max\{n,2\}$, $d\leq1$, $P\geq1$, and
$\mu(P)\leq|V|$, the last term is at most $Cd^2P\mu(P)$.
Adding the two bounds yields
\begin{equation}\label{wp-pure-tail}
\int_{\{X>KP\}}X\leq A_0d^2P\,|\{X>cP\}|,
\qquad P\geq P_b.
\end{equation}
The constants $A_0,K,c$ are fixed before $p,d,b$ are chosen.

For $t>0$, set $H_X(t)=\int_{\{X>t\}}X$ and $q=K/c>1$.
Since $t|\{X>t\}|\leq H_X(t)$, inequality~\eqref{wp-pure-tail}
becomes
\begin{equation}\label{wp-tail-contraction}
H_X(qt)\leq\theta H_X(t),\qquad
t\geq t_0:=cP_b,\qquad \theta=A_0d^2/c.
\end{equation}
Now fix $p>1$, choose $d$ so small that
$d\leq\min\{d_*,1,m\}$ and $\theta\leq q^{-p}$, and then
choose the smooth minorant $b$ in~\eqref{wp-minorant}.
The threshold $P_b$ is finite and independent of $j$.
The uniform bound for $\rho_u$ and~\eqref{wp-trace-mass} give
$\int_DX\leq M_X$ with $M_X$ independent of $j$.
For each $j$, $X$ is bounded. Integrating~\eqref{wp-tail-contraction}
and using Tonelli's formula gives
\[
\int_DX^p=(p-1)\int_0^\infty t^{p-2}H_X(t)\,dt
\leq M_X(qt_0)^{p-1}+\theta q^{p-1}\int_DX^p.
\]
Since $\theta q^{p-1}\leq q^{-1}<1$, this yields
$\int_DX^p\leq C(p,q)M_Xt_0^{p-1}$ uniformly in $j$.

On $U_1$,~\eqref{wp-uniform-gap} and $|D^2u_j|\leq\Delta u_j$ give
$\|D^2u_j\|_{L^p(U_1)}\leq C P_b^{1-1/p}$.
A subsequence of $D^2u_j$ converges weakly in $L^p(U_1)$.
Uniform convergence of $u_j$ identifies this limit with the
distributional Hessian of $u$. Interior interpolation gives
$u\in W^{2,p}(B_1)$, and weak lower semicontinuity together
with~\eqref{wp-modulus-threshold} yields
\begin{equation}\label{wp-modulus-estimate}
\|D^2u\|_{L^p(B_1)}
\leq C P_b^{1-1/p}
\leq C(1+\varepsilon^{-2})^{1-1/p}.
\end{equation}
Here $C,d$ depend only on $n,p,\lambda,\Lambda,A_*$, and
$\varepsilon$ is chosen from $\omega_f$ after $d$ is fixed.
Returning to the original variables gives
\[
\|D^2u_r\|_{L^p(B_1)}
=r^{-n/p}\|D^2u\|_{L^p(B_r(x_0))}.
\]
Moreover, $A_*=r^{-2}\osc_{B_{2r}(x_0)}u$, and the relevant modulus
is that of $f_r$ on $B_2$. This proves~\eqref{wp-interior-estimate}
and the local regularity assertion.
\end{proof}

\section{Remarks on interior \texorpdfstring{$C^{2,\alpha}$}{C2,alpha} regularity}

We revisit the interior $C^{2,\alpha}$ regularity theorem of
Chen, Zhou, and Zhu~\cite[Theorems~1.1--1.2]{CZZ}, using
Theorem~\ref{thm-w2p} and their second-difference estimate.

\begin{theorem}\label{cor-schauder}
Let $\Omega\subset\mathbb R^n$ be a domain, $n\geq2$, and let
$u\in C^0(\Omega)\cap\Phi_2(\Omega)$ satisfy
\[
\mu_2[u]=f\,dx,\qquad f\in C^\alpha_{\mathrm{loc}}(\Omega),\qquad
0<\lambda\leq f\leq\Lambda<\infty,
\]
where $0<\alpha<1$. Then $u\in C^{2,\alpha}_{\mathrm{loc}}(\Omega)$.
For every $B_{2r}(x_0)\Subset\Omega$,
\[
\|D^2u\|_{L^\infty(B_r(x_0))}
+r^\alpha[D^2u]_{C^\alpha(B_r(x_0))}\leq C,
\]
where $C$ depends only on $n$, $\alpha$, $\lambda$, $\Lambda$,
$r^{-2}\osc_{B_{2r}(x_0)}u$, and
$[\widetilde f]_{C^\alpha(B_2)}$, with
$\widetilde f(y)=f(x_0+ry)$.
\end{theorem}

\begin{proof}
By scaling and subtracting a constant, it suffices to work on $B_2$
with a fixed $L^\infty$ bound for $u$. We first consider smooth
admissible solutions for which the stated data are uniformly controlled.
Since $\omega_f(t)\leq[f]_{C^\alpha(B_2)}t^\alpha$,
Theorem~\ref{thm-w2p}, with $p=2n$, and Morrey's inequality give
a uniform bound $[Du]_{C^{0,1/2}(B_{3/4})}\leq K$.

Suppose, for contradiction, that there is a sequence of such solutions
$u_j$ with $\sup_{B_{1/2}}\Delta u_j\to\infty$.
By~\cite[Lemmas~2.2 and~2.4]{CZZ}, after passing to a subsequence,
there exist $M_j\to\infty$, unit vectors
$e_j$, points $z_j$, and scales
\[
h_j\asymp\frac{1}{M_j\log^2(e+M_j)},
\]
with comparison constants independent of $j$ and
$z_j,z_j\pm h_je_j\in B_{3/4}$ for all large $j$.
Their second-difference bound and the gradient H\"older estimate give
\begin{align*}
cM_j
&\leq\frac{u_j(z_j+h_je_j)+u_j(z_j-h_je_j)-2u_j(z_j)}{h_j^2}\\
&=\frac1{h_j^2}\int_0^{h_j}
\bigl(Du_j(z_j+te_j)-Du_j(z_j-te_j)\bigr)\cdot e_j\,dt\\
&\leq CK h_j^{-1/2}
\leq CK\sqrt{M_j}\log(e+M_j).
\end{align*}
This contradicts $M_j\to\infty$. Since $|D^2u|\leq\Delta u$
for admissible solutions, we obtain the Hessian bound on $B_{1/2}$.
A fixed covering and rescaling give the same bound on $B_{3/2}$.

The Hessian bound and $\sigma_2(D^2u)\geq\lambda$ keep the Hessians
in a fixed compact subset of $\Gamma_2$, on which
$G=\sqrt{\sigma_2}$ is smooth, concave, and uniformly elliptic.
The argument of~\cite[Section~4.2]{CZZ} now applies:
the supporting-plane extension of $G$ and the Evans--Krylov estimate
give a uniform $C^{2,\beta}$ bound for some $0<\beta<\alpha$;
the local perturbative Schauder estimate then gives a
$C^{2,\alpha}$ bound on $B_1$. All constants depend only on the
stated data.

For a general solution, apply Lemma~\ref{thm2.4} on $B_{7/4}$,
convolving $g=\sqrt f$ within $B_2$. The resulting densities
$f_j=g_j^2$ satisfy $\lambda\leq f_j\leq\Lambda$, with
$C^\alpha(B_{7/4})$ norms controlled by the stated data, and
$u_j\to u$ uniformly on $\overline{B}_{7/4}$.
The smooth estimates, applied on smaller balls, give a uniform
$C^{2,\alpha}(B_1)$ bound. Passing to a subsequence locally in $C^2$ and
using lower semicontinuity of the H\"older seminorm gives the
asserted estimate. Scaling back proves the theorem.
\end{proof}

\section*{Acknowledgements}

The author is deeply grateful to Professors Dongsheng Li and Yu Yuan for their guidance and support.
AI tools were used as an aid in preliminary mathematical exploration and in refining the exposition.
All mathematical arguments were independently verified by the author, who takes full responsibility for the manuscript.

\end{document}